\documentclass[12pt]{amsart}
\usepackage{latexsym}
\usepackage{amssymb} % \mathbb
\usepackage{mathrsfs}
\usepackage{amsmath}
\usepackage{latexsym}
\usepackage{delarray}
\usepackage{amssymb,amsmath,amsfonts,amsthm,%eucal,
mathrsfs}
\usepackage{hyperref}
\renewcommand\eqref[1]{(\ref{#1})} %Need with hyperref
\newtheorem{theorem}{Theorem}[section]

\newtheorem{lemma}[theorem]{Lemma}
\newtheorem{proposition}[theorem]{Proposition}

\newtheorem{remark}[theorem]{Remark}

\newcommand{\wt}[1]{\widetilde{#1}}

\newcommand{\Cinf}{\ensuremath{\mathcal{C}^\infty}}

\newcommand{\D}{\ensuremath{{\mathcal D}}}

\newcommand{\E}{\ensuremath{{\mathcal E}}}

\newcommand{\mb}[1]{\ensuremath{\mathbb{#1}}}
\newcommand{\N}{\mb{N}}

\newcommand{\R}{\mb{R}}
\newcommand{\C}{\mb{C}}

\newcommand{\lara}[1]{\langle #1 \rangle}

\newfont{\bl}{msbm10 scaled \magstep2}

\newcommand{\beq}{\begin{equation}}
\newcommand{\eeq}{\end{equation}}
\newcommand{\eps}{\varepsilon}

\renewcommand{\Re}{\ensuremath{\mathrm{Re}}}

\newcommand{\esp}{\mathrm{e}}
\newcommand\Rn{{\mathbb R}^n}

\renewcommand\N{{\mathbb N}_0}

\title[$C^\infty$ well-posedness of hyperbolic systems]{
On $C^\infty$ well-posedness of hyperbolic systems with multiplicities}

\author[Claudia Garetto]{Claudia Garetto}
\address{
  Claudia Garetto:
  \endgraf
  Department of Mathematical Sciences
  \endgraf
  Loughborough University
  \endgraf
  Loughborough, Leicestershire, LE11 3TU
  \endgraf
  United Kingdom
  \endgraf
  {\it E-mail address} {\rm c.garetto@lboro.ac.uk}
  }
\author[Michael Ruzhansky]{Michael Ruzhansky}
\address{
  Michael Ruzhansky:
  \endgraf
  Department of Mathematics
  \endgraf
  Imperial College London
  \endgraf
  180 Queen's Gate, London SW7 2AZ
  \endgraf
  United Kingdom
  \endgraf
  {\it E-mail address} {\rm m.ruzhansky@imperial.ac.uk}
  }

\thanks{The first author was supported by the
EPSRC First grant EP/L026422/1. The second author was supported in parts by the EPSRC
 grant EP/K039407/1 and by the Leverhulme Grant RPG-2014-02.
}
\date{}


\begin{document}

\maketitle

\begin{abstract}
In this paper we study first order hyperbolic systems with multiple characteristics (weakly hyperbolic) and time-dependent analytic coefficients. The main question is when the Cauchy problem for such systems is well-posed in $C^{\infty}$ and in $\D'$.
We prove that the analyticity of the coefficients combined with suitable hypotheses on the eigenvalues guarantee the
$C^\infty$ well-posedness of the corresponding Cauchy problem. This result is an extension to systems of the analogous results for scalar equations
recently obtained by Jannelli and Taglialatela in \cite{JT} and by the authors in \cite{GR:13}.
\end{abstract}

\section{Introduction}
This paper is devoted to hyperbolic systems of the type
\[
D_t u-A(t,D_x)u=0,\quad t\in[0,T],\, x\in\R^n,
\]
where $A$ is a $m\times m$ matrix of first order differential or pseudo-differential operators with $t$-analytic entries and the eigenvalues $\lambda_1(t,\xi), \lambda_2(t,\xi),\dots,\lambda_m(t,\xi)$  of the matrix $A(t,\xi)$ are real. In this case we say that the matrix $A(t,\xi)$ is hyperbolic.

It is well-known that the corresponding Cauchy problem
\beq
\label{CP_intro}
\begin{split}
D_t u-A(t,D_x)u&=0,\quad t\in[0,T],\, x\in\R^n,\\
u(0,x)&=g(x),
\end{split}
\eeq
is $C^\infty$-well-posed if the coefficients of the system are smooth and the eigenvalues of $A(t,\xi)$ are distinct (so \eqref{CP_intro} is strictly hyperbolic). In this case, also large time asymptotics are well studied even allowing fast oscillations in coefficients, see e.g.
\cite{RW:11} (and also an extended exposition of such problems in \cite{RW:14}).

At the same time, if we do not assume that all the eigenvalues are distinct, much less is known even if $A(t,\xi)$ is analytic in $t$. For example, if we assume that the characteristics (even $x$-dependent) are smooth and satisfy certain transversality relations, the $C^{\infty}$-well-posedness was shown in \cite{KR:07}. However, in the case of only time-dependent coefficients these transversality conditions are not satisfied. 

In general, in presence of multiplicities the well-posedness is usually expected to hold in Gevrey spaces even when the coefficients are analytic. For example, for the scalar equation
$$
\partial_{t}^{2}u-2t\partial_{t}\partial_{x}u+t^{2}\partial_{x}^{2}u=0
$$
in one space variable, the Cauchy problem is well-posed in the Gevrey class $\gamma^{s}$ for $s<2$ and ill-posed in $\gamma^{s}$ for $s>2$.

The first results of this type for $t$-dependent hyperbolic systems of size $2\times 2$ and $3\times 3$ have been obtained by d'Ancona, Kinoshita and Spagnolo in \cite{dAKS:04, dAKS:08}.
For $x$-dependent $2\times 2$ systems some results are also available, see e.g. \cite{GrR:13}.
Later, the former results have been extended to any matrix size by Yuzawa in  \cite{Yu:05} and to $(t,x)$-dependent coefficients jointly by Kajitani and Yuzawa in \cite{KY:06}. 
In such problems, the existing techniques apply equally well for equations with coefficients (or characteristics) of lower (e.g. H\"older) regularity.
More precisely, if the eigenvalues of $A$ are of H\"older order $\alpha\in(0,1]$ in $t$ and their multiplicity does not exceed $r$ then the Cauchy problem \eqref{CP_intro} with initial data in the Gevrey class $\gamma^s$ has a unique solution $u$ in $(C^1([0,T],\gamma^s(\R^n))^m$ provided that 
\beq
\label{s_Yu}
1\le s< 1+\frac{\alpha}{r}.
\eeq
In this direction, equations with even lower (e.g. distributional) regularity have been also considered, see e.g. \cite{GR:15} and also \cite{Gar:15}.

Recently, different authors have studied weakly hyperbolic scalar equations with analytic coefficients (see, for instance \cite{JT} and \cite{GR:13}) but systems have not been investigated from this point of view. Here, for the first time, we consider first order hyperbolic systems with analytic coefficients and multiple eigenvalues and we prove that under suitable conditions on the matrix $A$, formulated in terms of its eigenvalues, they are $C^\infty$-well-posed, in the sense that given initial data in $C^\infty$ the Cauchy problem \eqref{CP_intro} has a unique solution in $(C^1([0,T];C^\infty(\R^n))^m$.

Thus, it is the purpose of this paper to investigate under which conditions on the matrix $A$ the solution $u$ does actually belong to the space $C^1([0,T];C^\infty(\R^n))^m$. The main idea is an extension to systems of the previous works on higher order equations with analytic coefficients and lower order terms after a reduction to block Sylvester form.

\smallskip
More precisely, the analysis of this paper will consist of the following three steps:

\begin{itemize}
\item First, we make an observation (Theorem \ref{theo_yu_ext}) that the results of Yuzawa \cite{Yu:05}, and Kajitani and Yuzawa \cite{KY:06}, can be extended to produce the existence of some (ultradistributional) solution to the Cauchy problem \eqref{CP_intro}. It is then our task to improve its regularity to $C^{\infty}$ or to $\D'$ depending on the regularity of the Cauchy data.
This step is done in Section \ref{SEC:uw}.

\item Second, we consider matrices $A(t,D_{x})$ in Sylvester form and prove (in Theorem \ref{theo_main_H}) that in this case the Cauchy problem \eqref{CP_intro} is well-posed in 
$C^{\infty}$. This step in done in Section \ref{SEC:sf}.

\item Third, we extend the above to any weakly hyperbolic matrix $A$ or, in other words, we show that we can drop the assumption of Sylvester form for the matrix $A$. This is done by transforming a general $m\times m$ system 
\[
D_t-A(t,D_x)
\]
into the $m^{2}\times m^{2}$ block Sylvester system. This extended system will be still hyperbolic (in fact, the principal part will have the same eigenvalues), but such reduction will (unfortunately) produce some lower order terms. Therefore, we carry out a careful analysis of the appearing matrix of the lower order terms by considering the suitable Kovalevskian and hyperbolic energies in different frequency domains. This will yield the desired $C^\infty$-well-posedness as well as
the distributional well-posedness for the original Cauchy problem \eqref{CP_intro} in Theorem \ref{theo_main}. This analysis will be carried out in Section \ref{SEC:mr} and Section \ref{SEC:p}.
\end{itemize}

In Section \ref{SEC:ex} we illustrate the appearing Levi-type conditions in the example of $2\times 2$ systems. We also note that the obtained conditions can be expressed entirely in terms of the coefficients of the matrix $A(t,x)$ (rather than its eigenvalues) and are, therefore, computable. We refer to \cite{JT} and to \cite{GR:13} for the discussions of such expressions.

Finally we note that in problems concerning systems, it is often important whether the system can be diagonalised or whether it contains Jordan blocks, see e.g. \cite{KR:07} or
\cite{GrR:13}, for some respective results and further references. However, this is not an issue for the present paper since we are able to obtain the well-posedness results avoiding such assumptions.

\section{Preliminary results}
\label{SEC:1}

In this section we discuss several preliminary results needed for our analysis. First, we make an observation that the results of Yuzawa \cite{Yu:05}, and Kajitani and Yuzawa \cite{KY:06}, can be extended to produce the existence of an ultradistributional solution, thus enabling our further reductions. Then, we look at systems in the Sylvester form.

\subsection{Ultradistributional well-posedness}
\label{SEC:uw}

For convenience of the reader we recall Yuzawa's well-posedness result proven in \cite{Yu:05}. We begin by introducing for $\rho>0$ and $s>1$, the space $H^l_{\Lambda(\rho,s)}$ of all $f\in L^2(\R^n)$ such that
\[
\lara{\xi}^l\esp^{\Lambda(\rho,s)}\widehat{f}(\xi)\in L^2(\R^n_\xi),
\]
where $\Lambda(\rho,s)=\rho\lara{\xi}^{\frac{1}{s}}$. Let now the coefficients of the matrix $A$ be of class $C^\alpha$ and let $s$ be as in \eqref{s_Yu}. Theorem 1.1 in \cite{Yu:05} states that if the initial data $g$ has entries in $H^l_{\Lambda(T,s)}$ then the Cauchy problem \eqref{CP_intro} has a unique solution $u(t,x)$ 
such that $\esp^{(T-t)\lara{D_x}^{\frac{1}{s}}}u(t,x)\in (C([0,T];H^l))^m\cap (C^1([0,T];H^{l-1}))^m$, for $t\in[0,T]$ and $x\in\R^n$. From Lemma 1.2 in \cite{Kaj:83} by Kajitani one has that for any $f\in\gamma^s_c(\R^n)$ and $l\in\R$ there exists $\rho>0$ (depending on $f$) such that $f\in H^l_{\Lambda(\rho,s)}$ and conversely, if $f$ is a compactly supported element of some $H^l_{\Lambda(\rho,s)}$ then it is a compactly supported Gevrey function of order $s$. It then follows that the previous well-posedness results in $H^l_{\Lambda(\rho,s)}$ spaces can be formulated in Gevrey classes. More precisely, Theorem 1.2 in \cite{Yu:05} states that given initial data with entries in $\gamma^s_c(\R^n)$ for $s$ as in \eqref{s_Yu}, there exists a unique solution $u\in C^1([0,T]; \gamma^s(\R^n))^m$ of the Cauchy problem \eqref{CP_intro}.

Note that  the characterisation of Gevrey functions via weighted Sobolev spaces can be extended to Gevrey Beurling ultradistributions. We recall that $f\in C^\infty(\R^n)$ belongs to the Beurling Gevrey class $\gamma^{(s)}(\R^n)$ if for every compact set $K\subset\R^n$
and for every constant $A>0$ there exists a constant $C_A>0$ such that for all $\alpha\in\N^n$ the estimate
\[
|\partial^\alpha f(x)|\le C_A A^{|\alpha|}(\alpha!)^s
\]
holds uniformly in $x\in K$. The space $\D'_{(s)}(\R^n)$ of Gevrey Beurling ultradistributions is defined as the dual of $\gamma_c^{(s)}(\R^n)$ while the space of $\E'_{(s)}(\R^n)$ of compactly supported Gevrey Beurling ultradistributions is the dual of $\gamma^{(s)}(\R^n)$. In analogy to Gevrey classes one has that a real analytic functional $v$ belongs to $\mathcal{E}'_s(\R^n)$ if and only if for any $\nu>0$ there exists $C_\nu>0$ such that
\[
|\widehat{v}(\xi)|\le C_\delta\,\esp^{\nu\lara{\xi}^{\frac{1}{s}}}
\]
for all $\xi\in\Rn$, and similarly, $v\in \mathcal{E}'_{(s)}(\R^n)$ if and only if there exist $\nu>0$ and $C>0$ such that
\[
|\widehat{v}(\xi)|\le C\,\esp^{\nu\lara{\xi}^{\frac{1}{s}}}
\]
for all $\xi\in\Rn$ (see Proposition 13 in \cite{GR:11}). Combining these observations with Kajitani and Yuzawa's method in \cite{Yu:05} and \cite{KY:06} one can easily extend Lemma 1.2 in \cite{Kaj:83} and deduce the corresponding ultradistributional well-posedness results. More precisely, we have the following lemma and well-posedness theorems.

\begin{lemma}\label{Lemma:ud}
\leavevmode
\begin{itemize}
\item[(i)] For any $v\in \E'_{(s)}(\R^n)$ and $l\in\R$ there exists $\rho>0$ such that $v\in H^l_{-\Lambda(\rho,s)}$.
\item[(ii)] If  $v\in H^l_{-\Lambda(\rho,s)}$ is compactly supported then $v\in \E'_{(s)}(\R^n)$.
\end{itemize}
\end{lemma}

\begin{proof}
(i) From the Fourier characterisation of ultradistributions we have that there exist constants $c>0$ and $\rho>0$ such that 
\[
|\widehat{v}(\xi)|\le c\,\esp^{\rho\lara{\xi}^s},
\]
for all $\xi\in\R^n$. It follows that
\[
\lara{\xi}^l\esp^{-(\rho+1)\lara{\xi}^s}|\widehat{v}(\xi)|\le c\lara{\xi}^l\esp^{-\lara{\xi}^s},
\]
where the right-hand side is clearly an element of $L^2$. Thus, $v\in H^l_{-\Lambda(\rho,s)}$.

(ii) Let now $A(\R^n)$ be the set of analytic functions and $H^l_{-\Lambda(\rho,s)}$ be the set of all functionals $v$ on $A(\R^n)$ such that 
\beq
\label{est_1}
\lara{\xi}^l\esp^{-\Lambda(\rho,s)}\widehat{v}(\xi)\in L^2(\R^n_\xi).
\eeq
Assuming that $v$ is compactly supported we know that $\widehat{v}$ is an analytic function satisfying an estimate of the type
\beq
\label{est_2}
|\widehat{v}(\xi)|\le c\lara{\xi}^N,
\eeq
for some $c>0$ and $N\in\N$.  Since we can write \eqref{est_1} as 
\[
\lara{\xi}^l\esp^{-\Lambda(\rho,s)}\widehat{v}(\xi)=g(\xi),
\]
where $g\in L^2(\R^n)$, by using \eqref{est_2} we conclude that $|g(\xi)|\le c_1\esp^{-\rho_1\lara{\xi}^s}$ for some $c_1,\rho_1>0$. Hence, it follows that 
\[
|\widehat{v}(\xi)|\le c'\esp^{\rho\lara{\xi}^s}.
\]
This proves that $v$ is an ultradistribution in $\mathcal{E}'_{(s)}(\R^n)$. 
\end{proof}

We can now recall the precise form of Kajitani-Yuzawa result described earlier.

\begin{theorem}
\label{theo_yu_ext}
Let the coefficients of the matrix $A$ be of class $C^\alpha$ and let $A$ have real eigenvalues which do not exceed the multiplicity $r$ and let 
\[
1\le s< 1+\frac{\alpha}{r}.
\]
Then, for any initial data $g$ with entries in $H^l_{-\Lambda(T,s)}$ the Cauchy problem \eqref{CP_intro} has a unique solution $u(t,x)$ 
such that
\[
\esp^{-(T-t)\lara{D_x}^{\frac{1}{s}}}u(t,x)\in (C([0,T];H^l))^m\cap (C^1([0,T];H^{l-1}))^m,
\]
for $t\in[0,T]$ and $x\in\R^n$.
\end{theorem}

As a consequence of Lemma \ref{Lemma:ud} and Theorem \ref{theo_yu_ext}, we obtain the following ultradistributional well-posedness result which will be the starting point for our analysis.

\begin{theorem}
\label{theo_yu_ultra}
Under the hypotheses of Theorem \ref{theo_yu_ext} for any initial data $g$ with entries in $\E'_{(s)}(\R^n)$ the Cauchy problem \eqref{CP_intro} has a unique ultradistributional solution $u\in C^1([0,T];\D'_{(s)}(\R^n))^m$.
\end{theorem}

We now turn to a preliminary setting of Sylvester matrices.

\subsection{Systems in Sylvester form}
\label{SEC:sf}

From now on we concentrate on the Cauchy problem \eqref{CP_intro}
\[
\begin{split}
D_t u-A(t,D_x)u&=0,\quad t\in[0,T],\, x\in\R^n,\\
u(0,x)&=g(x),
\end{split}
\]
when the entries of the matrix $A$ are analytic in $t$. By applying Theorem \ref{theo_yu_ultra} we already know that if we take initial data in $(C^\infty_c(\R^n))^m$ then a solution $u$ exists in $C^1([0,T];\D'_{(s)}(\R^n))^m$. 

First, we briefly collect some preliminaries, for more details we refer the reader to \cite{GR:13,JT}.

Thus, here we assume that $A(t,\xi)$, the matrix of the principal part of the operator $D_t u-A(t,D_x)$, is a matrix of first order \emph{pseudo-differential operators} of Sylvester type (we will show in the next section that this assumption is not restrictive). It means that we can write $A(t,\xi)$ as $\lara{\xi}A_0(t,\xi)$, where  
\beq
\label{AJT}
A_0(t,\xi)=\left(
    \begin{array}{ccccc}
      0 & 1 & 0 & \dots & 0\\
      0 & 0 & 1 & \dots & 0 \\
      \dots & \dots & \dots & \dots & 1 \\
      h_m & h_{m-1} & \dots & \dots & h_1 \\
    \end{array}
  \right),
\eeq
for some $h_j$, $j=1,\dots, m$, symbols of order $0$ analytic in $t$.
The eigenvalues of $A_0(t,\xi)$ are exactly the eigenvalues of $A(t,\xi)$ scaled by a factor $\lara{\xi}^{-1}$, i.e., $\lara{\xi}^{-1}\lambda_j(t,\xi)$, $j=1,\dots,m$. Hence, they are symbols of order $0$ in $\xi$ analytic with respect to $t$.

Let us now fix $t$ and $\xi$ and treat $A_0$ as a matrix with constant entries. Since $A_0$ is hyperbolic we can construct a real symmetric semi-positive definite $m\times m$ matrix $Q$ such that 
\beq
\label{QA}
QA_0-A_0^\ast Q=0
\eeq
and
\[
\det Q =\prod_{1\le k<j\le m}\lara{\xi}^{-2}(\lambda_j-\lambda_k)^2.
\]
The matrix $Q$ is called the standard symmetriser of $A_0$. Its entries are fixed polynomials functions of $h_1,..., h_m$ (or, equivalently, they can be expressed via the eigenvalues of $A_0$) and it is weakly positive definite if and only if $A_0$ is weakly hyperbolic (see \cite{J:89, J:09}).

Let $Q_j$ be the principal $j\times j$ minor of $Q$ obtained by removing the first $m-j$ rows and columns of $Q$ and let $\Delta_j$ be its determinant. When $j=m$ we use the notations $Q$ and $\Delta$ instead of $Q_m$ and $\Delta_m$.  The following proposition shows how the hyperbolicity of $A_0$ (or equivalently of $A$) can be seen at the level of the symmetriser $Q$ and of its minors (see \cite{J:09}).
\begin{proposition}
\label{propJan}
\leavevmode
\begin{itemize}
\item[(i)] $A$ is strictly hyperbolic if and only if $\Delta_j>0$ for all $j=1,...,m$.
\item[(ii)] $A$ is weakly hyperbolic if and only if there exists $r<m$ such that
    \[
    \Delta=\Delta_{m-1}=...=\Delta_{r+1}=0
    \]
    and $\Delta_r>0,...,\Delta_1>0$. (In this case there are exactly $r$ distinct roots).
\end{itemize}
\end{proposition}
Clearly, when $t$ and $\xi$ vary in their domains, respectively, $\Delta_r$ becomes a symbol $\Delta_r(t,\xi)$ homogeneous of degree $0$ in $\xi$ and analytic in $t$. When $\Delta_r$ is not identically zero one can define the function
\[
\widetilde{\Delta}_r(t,\xi)=\Delta_r(t,\xi)+\frac{(\partial_t\Delta_r(t,\xi))^2}{\Delta_r(t,\xi)},
\]
which is as well a symbol of order $0$ in $\xi$ and analytic in $t$.  Note that if $t\mapsto\Delta(t,\xi)$ vanishes of order $2k$ at
a point $t'$ then $t\mapsto\widetilde{\Delta}(t,\xi)$ vanishes of order $2k-2$ at $t'$ (\cite{JT}). 

In analogy with the scalar equation case treated in \cite{JT} and \cite{GR:13} the energy estimate that we will use for the system $D_t u-A(t,D_x)u=0$, when $A$ is in Sylvester form, will make use of the quotient
\beq
\label{quo}
\lara{\partial_t QV,V}/\lara{QV,V}.
\eeq
As already observed in \cite{JT} and \cite{GR:13}, estimating the quotient $\lara{\partial_t QV,V}/\lara{QV,V}$ is equivalent to estimating
the roots of the generalised Hamilton-Cayley polynomial
\beq
\label{HamCay}
\det(\tau Q-\partial_t Q)=\sum_{j=0}^m d_j(t)\tau^{m-j}
\eeq
of $Q$ and $\partial_t Q$, where $d_0=\det Q$, $d_1=-\partial_t(\det Q)$, $d_m=(-1)^m \det (\partial_t Q)$ and, if $m\ge 2$, $d_2=\frac{1}{2}{\rm trace}(\partial_t Q\partial_t (Q^{{\rm co}}))$, where $Q^{{\rm co}}$ is the cofactor matrix of $Q$. From the identity 
\[
\tau_1^2+\cdots +\tau_m^2=\biggl(\frac{d_1}{d_0}\biggr)^2-2\frac{d_2}{d_0},
\]
valid for the roots $\tau_j$, $j=1,\dots,m$, of the generalised Hamilton-Cayley polynomial, we easily see that $d_2$ is crucial when estimating \eqref{quo}. Let
\[
\psi(t,\xi):=d_2(t,\xi)=\frac{1}{2}{\rm trace}(\partial_t Q\partial_t (Q^{{\rm co}})),
\]
and we call $\psi$ the \emph{check function} of $Q$.

\smallskip
For the moment we work under the following set (H) of hypotheses:
\begin{itemize}
\item[(i)] $A$ is a matrix of pseudo-differential operators of order $1$,
\item[(ii)] $A$ is in Sylvester form.
%\item[(iii)] $B$ has all row identically $0$ apart from the last one.
\end{itemize}
We can now state our preliminary well-posedness result for the Cauchy problem \eqref{CP_intro}. This result is obtained from Theorem 2.2 in \cite{GR:13} where the well-posedness in the scalar case is obtained by reduction to a first order pseudo-differential system with principal part in Sylvester form. Note that for technical reasons we will work on slightly bigger open interval $(\delta,T'+\delta)$ containing $[0,T]$.

\begin{theorem}
\label{theo_main_H}
Let $D_t -A(t,D_x)$ be the matrix operator in \eqref{CP_intro} under the hypotheses $(H)$. Let the entries of $A(t,D_x)$ be analytic in $t\in(\delta, T'+\delta)$ and let the matrix $A(t,\xi)$  be (weakly) hyperbolic. Let $Q(t,\xi)=\{q_{ij}(t,\xi)\}_{i,j=1}^{m}$
be the symmetriser of the matrix $A_0(t,\xi)=\lara{\xi}^{-1}A(t,\xi)$,
$\Delta$ its determinant and $\psi(t,\xi)$ its check function. Let
$\Delta(\cdot,\xi)\not\equiv 0$ in $(\delta, T'+\delta)$ for all $\xi$ with $|\xi|\ge 1$ and let $[0,T]\subset(\delta, T'+\delta)$.
%$\widetilde{\Delta}(t,\xi)=\Delta(t,\xi)+\frac{(\partial_t\Delta(t,\xi))^2}{\Delta(t,\xi)}$ with $\Delta(\cdot,\xi)\not\equiv 0$ for all $\xi\neq 0$ in $(\delta,T+\delta)$. %and let $B(t,\xi)$ be the matrix of the lower order terms with entries $b_i(t,\xi)$, $i=1,...,m$, as in \eqref{bi}. Finally, let
%$\Delta(\cdot,\xi)\not\equiv 0$ for all $\xi\neq 0$ in $(\delta,T+\delta)$.
Assume that there exists a constant $C_1>0$ such that
\beq
\label{GR1m}
|\psi(t,\xi)|\le C_1\widetilde{\Delta}(t,\xi)
\eeq
holds for all $t\in[0,T]$ and $|\xi|\ge 1$. Then the Cauchy problem 
\beq
\label{CP_1}
\begin{split}
D_t u-A(t,D_x)u&=0,\quad t\in[0,T],\, x\in\R^n,\\
u(0,x)&=g(x),
\end{split}
\eeq
is $\Cinf$ well-posed, in the sense that given $g\in (C^\infty(\R^n))^m$ there exists a unique solution $u$ in $C^\infty([0,T], C^\infty(\R^n))^m$, and it is also well-posed in $\mathcal D'(\Rn)$, i.e., for any $g\in (\mathcal{D}'(\R^n))^m$ there exists a unique solution $u\in C^\infty([0,T], \D'(\R^n))^m$.

\end{theorem}
For simplicity we will refer to the well-posedness above as $C^\infty$ well-posedness and distributional well-posedness in the interval $[0,T]$. Note that by the energy estimates we obtain first that the solution is $C^1$ with respect to $t\in[0,T]$ and then, by iterated differentiation in the original system, we conclude that the dependence in $t$ is actually $C^\infty$.

Our next aim is to extend the theorem above to any weakly hyperbolic matrix $A$, or in other words to drop the assumption of Sylvester form for the matrix $A$. This will be done by reducing a general system 
\[
D_t-A(t,D_x)
\]
into block Sylvester form. Unfortunately, this will produce some lower order terms and therefore a careful analysis of the new matrix $B$ of the lower order terms will be needed to achieve $C^\infty$ and distributional well-posedness. This will be done in the next sections.

 \section{Main result}
 \label{SEC:mr}
 
We perform a reduction to block Sylvester form of the system in \eqref{CP_intro} by following the ideas of d'Ancona and Spagnolo in \cite{DS}. We begin by considering the cofactor matrix $L(t,\tau,\xi)$ of $(\tau I-A(t,\xi))^T$ where $I$ is the $m\times m$ identity matrix. By applying the corresponding operator $L(t,D_t,D_x)$ to \eqref{CP_intro} we transform the system
\[
D_t u-A(t,D_x)u=0
\]
into
\beq
\label{red_1}
\mu(t,D_t,D_x)Iu-C(t,D_t,D_x)u=0,
\eeq
where $\mu(t,\tau,\xi)={\rm det}(\tau I-A(t,\xi))$ and $C(t,D_t,D_x)$ is the matrix of  lower order terms (differential operators of order $m-1$).  More precisely, $\mu(t,D_t,D_x)$ is an operator of the form
\[
\mu(t,D_t,D_x)=D_t^m+\sum_{h=0}^{m-1}b_{m-h}(t,D_x)D_t^h,
\]
with $b_{m-h}(t,\xi)$ a homogeneous polynomial of order $m-h$.

We now transform this set of scalar equations of order $m$ into a first order system of size $m^2\times m^2$ of pseudo-differential equations, by setting
\[
U=\{D_j^{j-1}\lara{D_x}^{m-j}u\}_{j=1,2,\dots,m},
\]
where $\lara{D_x}$ is the pseudo-differential operator with symbol 
$\lara{\xi}=(1+|\xi|^{2})^{1/2}$. We can therefore write \eqref{red_1} in the form
\beq
\label{red_2}
D_tU-\mathcal{A}(t,D_x)U+\mathcal{L}(t,D_x)U=0,
\eeq
where $\mathcal{A}$ is a $m^2\times m^2$ matrix made of $m$ identical blocks of the type
\begin{multline}
\label{block_A}
\lara{D_x}\cdot\\
\left(\begin{array}{cccccc}
                              0 & 1 & 0 & \cdots & 0 \\
                               0 & 0 & 1 &  \cdots & 0\\
                               \vdots & \vdots & \vdots & \vdots & \vdots\\
                              -b_{m}(t,D_x)\lara{D_x}^m& -b_{m-1}(t,D_x)\lara{D_x}^{-m+1} & \cdots  & \cdots & -b_1(t,D_x)\lara{D_x}^{-1}\\
                              \end{array}
                           \right),
                           \end{multline}
  with $b_{j}(t,D_x)$ a pseudo-differential operator of order $j$, $j=1,\dots,m$, analytic in $t$,                      
and the matrix $\mathcal{L}$  of the lower order terms is made of $m$ blocks of  size $m\times m^2$ of the type
\[
\left(\begin{array}{cccccc}
                              0 & 0 & 0 & \cdots & 0 & 0 \\
                               0 & 0 & 0 &  \cdots &  0 & 0\\
                               \vdots & \vdots & \vdots & \vdots & \vdots\\
                              l_{i,1}(t, D_x)& l_{i,2}(t,D_x) & \cdots  & \cdots & l_{i,m^2-1}(t,D_x) & l_{i,m^2}(t,D_x)
                                                     \end{array}
                           \right), 
\]
with $i=1,\dots,m$.  Note that the operators $l_{i,j}$, $j=1,\dots, m^2$, are all of order $0$ in $\xi$. 
%, and finally the right-hand side $\mathcal{R}$ is a $m^2\times 1$ matrix with $m$ column blocks of size $m\times 1$ of the type
%\[
%\left(\begin{array}{c}
    %                          0 \\
        %                       0 \\
            %                   \vdots \\
                %              r_j(t,x),
                    %                                 \end{array}
  %                         \right)
%\]
%$j=1,\dots,m$.
Hence, by construction the matrices $\mathcal{A}$ and $\mathcal{L}$ are made by pseudo-\-dif\-fe\-ren\-tial operators of order $1$ and $0$, respectively. 
Concluding, the Cauchy problem \eqref{CP_intro} has been now transformed into
\beq
\label{CP_syst_Syl}
\begin{split}
D_tU-\mathcal{A}(t,D_x)U+\mathcal{L}(t,D_x)U&=0,\\
U_{t=0} &= \{D_t^{j-1}\lara{D_x}^{m-j}g_0\}_{j=1,2,\dots,m}.
\end{split}
\eeq
This is a Cauchy problem of first order pseudo-differential equations with principal part in block Sylvester form. The size of the system has increased from $m\times m$ to $m^2\times m^2$ but the system is still hyperbolic, since the eigenvalues of any block of $\mathcal{A}(t,\xi)$ are the eigenvalues of the matrix $\lara{\xi}^{-1}A(t,\xi)$.

We now want to analyse the matrix $\mathcal{L}$ in more detail and study its relationship with the principal matrix $A$. For this purpose we observe that by definition of the operator $L(t,D_t,D_x)$ we have that 
\[
L(t,D_t,D_x)=\sum_{h = 0}^{m-1} {A}_{h}(t,D_x) D_t^{m-1-h},
\]
where
\beq
\label{formula_A_h}
{A}_{h}(t,D_x) = (-1)^m \sum_{h'=0}^{h} \sigma_{h'}^{(m)}(\lambda) A^{h-h'}(t,D_x),
\eeq
with $\lambda=(\lambda_1,\dots,\lambda_m)$, $$\sigma_{h'}^{(m)}(\lambda)=(-1)^{h'}\sum_{1\le i_1<...<i_{h'}\le m}\lambda_{i_1}...\lambda_{i_{h'}}$$ and $\sigma_0^{(m)}(\lambda)=1$. We can now prove the following linear algebra lemma.
\begin{lemma}
\label{chri_lemma}
The  entries of the matrix $\mathcal{L}$ of the lower order terms are of the type 
\[
\lara{\xi}^{-1}\sum_{k=1}^{m-1} c_{i(k),j(k)}(t)D_t^k a_{i(k),j(k)}(t,\xi), 
\]
where $1\le i(k),j(k)\le m$ and $c_{i(k),j(k)}$ is a bounded function in $t$.
\end{lemma}
\begin{proof}
We apply the operator $L(t,D_t,D_x)$ to $D_tI-A(t,D_x)$. By direct computations and by formula \eqref{formula_A_h} we have that
\begin{multline}
\label{comp_1}
L(t,D_t,D_x)(D_tI-A(t,D_x))=\sum_{h =0}^{m-1} {A}_{h}(t,D_x) D_t^{m-h}\\
- \sum_{h =0}^{m-1} {A}_{h}(t,D_x) \sum_{q=0}^{m-1-h} \binom{m-1-h}{q}D_t^{q}A(t,D_x)D_t^{m-1-h-q}
\end{multline}
By now writing the last term in \eqref{comp_1} as $-X-Y$, where 
\[
X=\sum_{h =0}^{m-1} {A}_{h}(t,D_x) A(t,D_x)D_t^{m-1-h}
\]
and
\[
Y=\sum_{h =0}^{m-1} {A}_{h}(t,D_x)\sum_{q=1}^{m-1-h}\binom{m-1-h}{q}D_t^{q}A(t,D_x)D_t^{m-1-h-q}
\]
we easily see that $\sum_{h =0}^{m-1} {A}_{h}(t,D_x) D_t^{m-h}-X=\mu(t,D_t,D_x)$, i.e. the principal part of the operator $L(t,D_t,D_x)(D_tI-A(t,D_x)$ while the lower order terms $C(t,D_t,D_x)$ are given by $-Y$. Hence,  
\[
C(t,D_t,D_x)=\sum_{h =0}^{m-1} {A}_{h}(t,D_x)\sum_{q=1}^{m-1-h}\binom{m-1-h}{q}D_t^{q}A(t,D_x)D_t^{m-1-h-q}.
\]
Note that $A_h$ contains only powers of the operator $A$ up to order $h$ and therefore $C$ contains powers of $A$ up to order $m-1$ and derivatives of $A$ from order $1$ to order $m-1$. Passing now to the reduction to a first order system of size $m^2\times m^2$ of pseudo-differential operators, we easily see that the entries of the matrix $\mathcal{L}$ in \eqref{red_2} are obtained by the matrix $C$ and therefore from $A_hD_t^{q}A$ suitably reduced to order $0$, i.e., $\lara{\xi}^{-h}A_h(t,\xi)D_t^{q}A(t,\xi)\lara{\xi}^{-1}$. Since $\lara{\xi}^{-h}A_h(t,\xi)$ is bounded with respect to $t$ and $\xi$ and $1\le q\le m-1$  we conclude that the entries of the matrix $\mathcal{L}$ are of the desired type.
\end{proof}

The representation formula in Lemma \ref{chri_lemma} implies the following estimate.

\begin{proposition}
\label{prop_matrix_L}
The matrix $\mathcal{L}$ is bounded by the derivatives of the matrix $A_0=\lara{\xi}^{-1}A$ up to order $m-1$, i.e., there exists a constant $c>0$ such that 
\beq
\label{est_L}
\Vert \mathcal{L}(t,\xi)\Vert \le c\max_{k=1, \dots, m-1}\Vert D_t^k A_0(t,\xi)\Vert,
\eeq
for all $t\in [0,T]$ and $\xi\in\R^n$, where $\Vert\cdot\Vert$ denotes the standard matrix norm.
\end{proposition}
We can now state our main result, which extends Theorem \ref{theo_main_H} to a general hyperbolic matrix $A$.
\begin{theorem}
\label{theo_main}
Let $D_t -A(t,D_x)$ be the matrix operator in \eqref{CP_intro}. Let the entries of $A(t,D_x)$ be analytic in $t\in(\delta, T'+\delta)$ and let the matrix $A(t,\xi)$  be (weakly) hyperbolic. Let $Q(t,\xi)=\{q_{ij}(t,\xi)\}_{i,j=1}^{m}$
be the symmetriser of the matrix $A_0(t,\xi)=\lara{\xi}^{-1}A(t,\xi)$,
$\Delta$ its determinant and $\psi(t,\xi)$ its check function. Let
$\Delta(\cdot,\xi)\not\equiv 0$ in $(\delta, T'+\delta)$ for all $\xi$ with $|\xi|\ge 1$ and let $[0,T]\subset(\delta, T'+\delta)$.
Assume that there exists a constant $C>0$ such that
\beq
\label{GR1m}
|\psi(t,\xi)|\le C\widetilde{\Delta}(t,\xi)
\eeq
and
\beq
\label{GRLevi}
\max_{k=1,\dots, m-1}\Vert \partial_t^k A_0(t,\xi)\Vert\le C({\Delta}(t,\xi)+\partial_t\Delta(t,\xi))
\eeq
for all $t\in[0,T]$ and $|\xi|\ge 1$. Then the Cauchy problem 
\beq
\label{CP_2}
\begin{split}
D_t u-A(t,D_x)u&=0,\quad t\in[0,T],\, x\in\R^n,\\
u(0,x)&=g(x),
\end{split}
\eeq
is $\Cinf$ well-posed and distributionally well-posed in $[0,T]$.   

\end{theorem}

Before proceeding with the energy estimate which will allow us to prove Theorem \ref{theo_main} we focus on the case $m=2$. The following explanatory example will help the reader to better understand the meaning of the hypotheses \eqref{GR1m} and \eqref{GRLevi}.

\subsection{Example: the case $m=2$}
\label{SEC:ex}

We recall that if $\lambda_1,\lambda_2$ are the eigenvalues of $A$ then 
\[
Q(t,\xi)=\left(\begin{array}{cc}
                              \lara{\xi}^{-2}(\lambda_1^2+\lambda_2^2)(t,\xi) & -\lara{\xi}^{-1}(\lambda_1+\lambda_2)(t,\xi)\\
                             -\lara{\xi}^{-1}(\lambda_1+\lambda_2) (t,\xi)& 2 \\
                              \end{array}
                           \right),
\]
with
\[
\Delta=\lara{\xi}^{-2}(\lambda_1-\lambda_2)^2(t,\xi)
\]
and
\[
\wt{\Delta}=\lara{\xi}^{-2}(\lambda_1-\lambda_2)^2(t,\xi)+2\lara{\xi}^{-2}(\partial_t\lambda_1-\partial_t\lambda_2)^2(t,\xi),
\]
and 
\[
\psi(t,\xi)=\frac{1}{2}{\rm trace}(\partial_t Q\partial_t (Q^{{\rm co}}))(t,\xi)=-\lara{\xi}^2(\partial_t\lambda_1+\partial_t\lambda_2)^2(t,\xi).
\]
It follows that in this case the hypothesis \eqref{GR1m} looks like
\[
(\partial_t\lambda_1+\partial_t\lambda_2)^2(t,\xi)\le C((\lambda_1-\lambda_2)^2(t,\xi)+(\partial_t\lambda_1-\partial_t\lambda_2)^2(t,\xi))
\]
and \eqref{GRLevi} is given by  
\[
\Vert \partial_t A_0(t,\xi)\Vert\le C\lara{\xi}^{-2}((\lambda_1-\lambda_2)^2(t,\xi)+|(\lambda_1-\lambda_2)(t,\xi)(\partial_t\lambda_1-\partial_t\lambda_2)(t,\xi)|).
\]
Note that when the matrix $A$ is already in Sylvester form the formulation of the hypotheses \eqref{GR1m} and \eqref{GRLevi} is simplified and sometimes trivial. For instance, when 
\[
A(t,\xi)=\xi\left(\begin{array}{cc}
                             0 &1\\
                             a^2(t) & 0\\
                              \end{array}
                           \right),
\]
$\xi\in\R$, both the hypotheses \eqref{GR1m} and \eqref{GRLevi} are trivially satisfied. Indeed, $\lambda_1(t,\xi)=-|a(t,\xi)|$ and $\lambda_2(t,\xi)=|a(t)\xi|$. This implies \eqref{GR1m} because $\psi(t,\xi)\equiv 0$ and \eqref{GRLevi} becomes
\[
|2a(t)a'(t)|\le C (4a^2(t)+4|a(t)a'(t)|,
\]
which is trivially true.

 \section{Proof of the main theorem}
 \label{SEC:p}
 
The proof of Theorem \ref{theo_main} is partly based on the analogous result for scalar equations in \cite{GR:13} to which we will refer for the complete details of some steps of the proof. This is due to the reduction to block Sylvester form explained in the previous section which allows to define the block diagonal $m^2\times m^2$-symmetriser
 \[
 \mathcal{Q}=\left(\begin{array}{cccc}
 Q & 0 & \cdots & 0\\
 0 & Q & \cdots & 0\\
 \vdots & \vdots & \vdots &\vdots\\
 0 & \cdots & \cdots & Q
     \end{array}
                           \right),
 \]
 where $Q$ is the symmetriser of the matrix $A_0=\lara{\xi}^{-1}A$. Since the reduction to block Sylvester form transforms the original system 
 \[
 D_t u-A(t,D_x)u=0
 \]
into the system
\[
D_tU-\mathcal{A}(t,D_x)U+\mathcal{L}(t,D_x)U=0,
\] 
with $\mathcal{A}$ in the block Sylvester form,
our proof will need to take care of the lower order terms in $\mathcal{L}$ which do not enter into $\mathcal{A}$. This will be done by using the Levi conditions introduced in \cite {GR:13} and in particular by referring to 
Remark 4.8 in \cite{GR:13}.

We begin by recalling some technical lemmas which have been proved in \cite{GR:13} and \cite{JT} which will be useful for our analysis of systems as well.

\begin{lemma}
\label{lemmaGR}
Let $Q(t,\xi)$ be the symmetriser of the weakly hyperbolic matrix $A(t,\xi)$ defined above. Then, there exist two positive constants $c_1$ and $c_2$ such that
\[
c_1\det Q(t,\xi)|V|^2\le \lara{Q(t,\xi)V,V}\le c_2|V|^2
\]
holds for all $t\in[0,T]$, $\xi\in\R^n$ and $V\in\C^m$.
\end{lemma}
 
\begin{lemma}
\label{lemmaEGR}
Let $Q(t,\xi)$ be the symmetriser of the matrix $A(t,\xi)$. Let $\Delta(t,\xi)=\det Q(t,\xi)$, $\wt{\Delta}(t,\xi)=\Delta(t,\xi)+(\partial_t \Delta(t,\xi))^2/\Delta(t,\xi)$, $\psi(t,\xi)$ the check function of $Q(t,\xi)$. Let $I$ be a closed interval of $\R$.
% and $W$ be a subset of $\R^n$ such that $\Delta(t,\xi)\neq 0$ for all $t\in I$ and $\xi\in W$.
Then,
\beq
\label{E1}
\sqrt{\frac{\Delta(t,\xi)}{\wt{\Delta}(t,\xi)}}\frac{\lara{\partial_t Q(t,\xi)V,V}}{\lara{Q(t,\xi)V,V}}\in L^\infty (I\times \R^n\times\C^m\setminus  0)
\eeq
if and only if
\beq
\label{C1}
\frac{\psi(t,\xi)}{\wt{\Delta}(t,\xi)}\in L^\infty (I\times \R^n).
\eeq
\end{lemma}
\begin{remark}
\label{rem_lemmas}
It is clear that Lemma \ref{lemmaGR} and Lemma \ref{lemmaEGR} are valid also for the block diagonal matrix $\mathcal{A}$ and the corresponding symmetriser $\mathcal{Q}$ as defined at the beginning of this section.
\end{remark}

\begin{lemma}
\label{lemmaJT}
 Let $\Delta(t,\xi)$ be the determinant of $Q(t,\xi)$ defined as above. Suppose that $\Delta(t,\xi)\not\equiv 0$. Then,
 \begin{itemize}
 \item[(i)] there exists $X\subset\mathbb{S}^{n-1}$ such that $\Delta(t,\xi)\not\equiv 0$ in $(\delta, T'+\delta)$ for any ${\xi}\in X$ and the set $\mathbb{S}^{n-1}\setminus X$ is negligible with respect to the Hausdorff $(n-1)$-measure;
 \item[(ii)] for any $[0,T]\subset(\delta, T'+\delta)$ there exist $c_1,c_2>0$ and $p,q\in\N$ such that for any $\xi\in X$ and any $\eps\in(0,\esp^{-1}]$ there exists $A_{\xi,\eps}\subset[a,b]$ such that:
\begin{itemize}
\item[-] $A_{\xi,\eps}$ is a union of at most $p$ disjoint intervals,
\item[-] $meas(A_{\xi,\eps})\le\eps$,
\item[-] $\min_{t\in[0,T]\setminus A_{\xi,\eps}} \Delta(t,\xi)\ge c_1\eps^{2q}\Vert\Delta(\cdot,\xi)\Vert_{L^\infty([0,T])}$,
\item[-]
\[
\int_{t\in[0,T]\setminus A_{\xi,\eps}}\frac{|\partial_t\Delta(t,\xi)|}{\Delta(t,\xi)}\, dt\le c_2\log\frac{1}{\eps}.
\]
\end{itemize}

\end{itemize}
\end{lemma}

To prove the $\Cinf$ well-posedness of the Cauchy problem \eqref{CP_intro} in the reduced form \eqref{CP_syst_Syl} we first apply the Fourier transform in $x$ and work on the equivalent system
\beq
\label{CP_syst_F}
\begin{split}
D_tV-\mathcal{A}(t,\xi)V+\mathcal{L}(t,\xi)V&=0,\\
V_{t=0} &= \{D_t^{j-1}\lara{\xi}^{m-j}\widehat{g_0}(\xi)\}_{j=1,2,\dots,m},
\end{split}
\eeq
where $V=\mathcal{F}_{x\to\xi}U(t,\cdot)(\xi)$. We then consider the energy
\[
E(t,\xi)=\begin{cases}
|V(t,\xi)|^2 & \text{for $t\in A_{\xi/|\xi|,\eps}$ and $\xi/|\xi|\in X$},\\
\lara{\mathcal{Q}(t,\xi)V(t,\xi),V(t,\xi)} & \text{for $t\in[a,b]\setminus A_{\xi/|\xi|,\eps}$ and $\xi/|\xi|\in X$,}
\end{cases}
\]
defined for $t\in[0,T]$, $\xi\in\R^n$ with $\xi/|\xi|\in X$, and $\eps\in(0,\esp^{-1}]$.
Note that $\Delta(t,\xi)>0$ when $t\in[0,T]\setminus A_{\xi/|\xi|,\eps}$ and $\xi/|\xi|\in X$,
and, thanks to Lemma \ref{lemmaJT},  $[0,T]\setminus A_{\xi/|\xi|,\eps}$ is a finite union of at most $p$ closed intervals $[c_i,d_i]$.
Moreover, the set $A_{\xi/|\xi|,\eps}$ is a finite union of open intervals whose total length does not
exceed $\eps$. 

We now define a Kovalevskian energy on $A_{\xi/|\xi|,\eps}$ and a hyperbolic energy on the complement.   

\subsection{The Kovalevskian energy}
Let $t\in[t',t^{''}]\subseteq A_{\xi/|\xi|,\eps}$ and $\xi/|\xi|\in X$. Hence
\[
\begin{split}
\partial_t E(t,\xi)&= 2\Re\lara{V(t,\xi),\partial_t V(t,\xi)}\\
& = 2\Re\lara{V(t,\xi),i\lara{\xi}\mathcal{A}(t,\xi)V(t,\xi)+i\mathcal{L}(t,\xi)V(t,\xi)}\le 2(c_\mathcal{A}\lara{\xi}+c_\mathcal{L})E(t,\xi).
\end{split}
\]
By Gronwall's Lemma on $[t',t^{''}]$ we get
\beq
\label{estEk}
|V(t,\xi)|\le \esp^{(c_\mathcal{A}\lara{\xi}+c_\mathcal{L})(t-t')}|V(t',\xi)|\le c\,\esp^{c\lara{\xi}(t-t')}|V(t',\xi)|.
\eeq

\subsection{The hyperbolic energy}
Let us work on any subinterval $[c_i,d_i]$ of $[0,T]\setminus A_{\xi/|\xi|,\eps}$. Assuming $\xi/|\xi|\in X$,
we have that $\Delta(t,\xi)>0$ on $[c_i,d_i]$. By definition of the symmetriser, we have that 
\[
\begin{split}
&\partial_t E(t,\xi)=\lara{\partial_t \mathcal{Q}(t,\xi)V(t,\xi),V(t,\xi)}\\
&+\lara{\mathcal{Q}(t,\xi)\partial_tV(t,\xi),V(t,\xi)}+\lara{\mathcal{Q}(t,\xi)V(t,\xi),\partial_t V(t,\xi)}\\
&=\frac{\lara{\partial_t \mathcal{Q}(t,\xi)V(t,\xi),V(t,\xi)}}{\lara{\mathcal{Q}(t,\xi)V,V}}E(t,\xi)+\lara{\mathcal{Q}(t,\xi)(i\lara{\xi}\mathcal{A}(t,\xi)+i\mathcal{L}(t,\xi))V(t,\xi),V(t,\xi)}\\
&+\lara{\mathcal{Q}(t,\xi)V(t,\xi),(i\lara{\xi}\mathcal{A}(t,\xi)+i\mathcal{L}(t,\xi))V(t,\xi)}\\
&=\frac{\lara{\partial_t \mathcal{Q}(t,\xi)V(t,\xi),V(t,\xi)}}{\lara{\mathcal{Q}(t,\xi)V,V}}E(t,\xi)+i\lara{(\mathcal{Q}(t,\xi)\mathcal{L}(t,\xi)-\mathcal{L}^\ast(t,\xi)\mathcal{Q}(t,\xi))V(t,\xi),V(t,\xi)}.
%&\le C\sqrt{\frac{\wt{\Delta}(t,\xi)}{\Delta(t,\xi)}}E(t,\xi)+
%|\lara{(Q(t,\xi)B(t,\xi)-B^\ast(t,\xi)Q(t,\xi))V(t,\xi),V(t,\xi)}|.\\
%&\le C\biggl(1+\frac{|\partial_t\Delta(t,\xi)|}{\Delta(t,\xi)}\biggr)E(t,\xi)+c_Q\frac{|B(t,\xi)|}{\Delta(t,\xi)}E(t,\xi).
\end{split}
\]
Now, by Lemma \ref{lemmaGR} and Lemma \ref{lemmaEGR}, the hypothesis \eqref{GR1m} implies that the quantity $$\frac{\lara{\partial_t \mathcal{Q}(t,\xi)V(t,\xi),V(t,\xi)}}{\lara{\mathcal{Q}(t,\xi)V,V}}$$ is bounded by $$\sqrt{\frac{\wt{\Delta}(t,\xi)}{\Delta(t,\xi)}}.$$ Hence, by definition of $\widetilde{\Delta}$ we conclude that
\begin{multline}
\label{estE1}
\partial_t E(t,\xi) \\ 
\le C\sqrt{\frac{\wt{\Delta}(t,\xi)}{\Delta(t,\xi)}}E(t,\xi)+|\lara{(\mathcal{Q}(t,\xi)\mathcal{L}(t,\xi)-\mathcal{L}^\ast(t,\xi)\mathcal{Q}(t,\xi))V(t,\xi),V(t,\xi)}|\\
\le  C\biggl(1+\frac{|\partial_t\Delta(t,\xi)|}{\Delta(t,\xi)}\biggr)E(t,\xi)+|\lara{(\mathcal{Q}(t,\xi)\mathcal{L}(t,\xi)-\mathcal{L}^\ast(t,\xi)\mathcal{Q}(t,\xi))V(t,\xi),V(t,\xi)}|.
\end{multline}
We now have to deal with the lower order terms. By arguing as in Remark 4.8 in \cite{GR:13} we can estimate
\[
|\lara{(\mathcal{Q}(t,\xi)\mathcal{L}(t,\xi)-\mathcal{L}^\ast(t,\xi)\mathcal{Q}(t,\xi))V(t,\xi),V(t,\xi)}|\le  c\Vert \mathcal{L}\Vert|V|^2+c\Vert \mathcal{L}^\ast\Vert |V|^2.
\]
The hypothesis \eqref{GRLevi} combined with Proposition \ref{prop_matrix_L} implies that both $\Vert\mathcal{L}\Vert$ and $\Vert \mathcal{L}^\ast\Vert$ are bounded by
\[
\Delta(t,\xi)+|\partial_t\Delta(t,\xi)|.
\]
Hence, by applying Lemma \ref{lemmaGR} we arrive at the estimate
\begin{multline}
\label{estE2}
|\lara{(\mathcal{Q}(t,\xi)\mathcal{L}(t,\xi)-\mathcal{L}^\ast(t,\xi)\mathcal{Q}(t,\xi))V(t,\xi),V(t,\xi)}|
\\ \le C'\biggl(\frac{\Delta(t,\xi)+|\partial_t\Delta(t,\xi)|}{\Delta(t,\xi)}\biggr)E(t,\xi)\\
\le  C'\biggl(1+\frac{|\partial_t\Delta(t,\xi)|}{\Delta(t,\xi)}\biggr)E(t,\xi).
\end{multline}
Finally, by combining \eqref{estE1} and \eqref{estE2} we obtain the final energy estimate
\beq
\label{estE3}
\partial_t E(t,\xi)\le  c'\biggl(1+\frac{|\partial_t\Delta(t,\xi)|}{\Delta(t,\xi)}\biggr)E(t,\xi).
\eeq

\subsection{Completion of the proof}
 
We are now ready to prove Theorem \ref{theo_main}.
\begin{proof}[Proof of Theorem \ref{theo_main}]
We begin by observing that, by the finite speed of propagation for hyperbolic equations, we can always assume that the Cauchy data in \eqref{CP_intro} are compactly supported. We refer to the Kovalevskian energy and the hyperbolic energy introduced above. We note that in the energies under consideration we can assume $|\xi|\ge 1$ since the continuity of $V(t,\xi)$ in $\xi$ implies that both energies are bounded for $|\xi|\le 1$. Let us consider the hyperbolic energy on the interval $[c_i,d_i]$. By Gronwall's Lemma on $[c_i,d_i]$ we get the inequality
\beq
\label{estEhyp}
E(t,\xi)\le \esp^{c(d_i-c_i)}\exp\biggl(c\int_{c_i}^t\frac{|\partial_s\Delta(s,\xi)|}{\Delta(s,\xi)}\,ds\biggr)E(c_i,\xi).
\eeq
By Lemma \ref{lemmaJT}, (ii), we have
\[
\Delta(t,\xi)\ge \min_{s\in [a,b]\setminus A_{\xi,\eps}}\Delta(s,\xi)\ge c_1\eps^{2q}\Vert \Delta(\cdot,\xi)\Vert_{L^\infty([a,b])},
\]
for all $t\in[c_i,d_i]$. Hence, applying Lemma \ref{lemmaGR} to \eqref{estEhyp}
%\[
%\Delta(t,\xi)\ge\min_{t\in[c_i,d_i], |\eta|=1}\ge \min_{t\in [a,b]\setminus A_{\xi,\eps}}\Delta(t,\xi)
%\]
we have that there exists a constant $C>0$ such that
\begin{multline}
\label{estEhyp2}
|V(t,\xi)|^2\le C\frac{1}{\eps^{2q}\Vert \Delta(\cdot,\xi)\Vert_{L^\infty([a,b])}}\exp\biggl(\int_{c_i}^t
\frac{|\partial_s\Delta(s,\xi)|}{\Delta(s,\xi)}\,ds\biggr)|V(c_i,\xi)|^2,\\
\le C\frac{1}{\eps^{2q}\Vert \Delta(\cdot,\xi)\Vert_{L^\infty([a,b])}}\esp^{C\log(1/\eps)}|V(c_i,\xi)|^2,
\end{multline}
for all $t\in[c_i,d_i]$ and for $|\xi|\ge 1$. Note that in the estimate above we have used Lemma \ref{lemmaJT}, (ii), in the last step.
Since the number of the closed interval $[c_i,d_i]$ does not exceed $p$, a combination of the Kovalevskian energy \eqref{estEk} with the hyperbolic energy \eqref{estEhyp2} leads to\[
|V(b,\xi)|\le C\frac{1}{\eps^{pq}\Vert \Delta(\cdot,\xi)\Vert^{p/2}_{L^\infty([a,b])}}\esp^{C(\log(1/\eps)+\eps|\xi|)}|V(a,\xi)|,
\]
for $|\xi|\ge 1$. At this point setting $\eps=\esp^{-1}\lara{\xi}^{-1}$ we have that there exist constants $C'>0$ and $\kappa\in\N$ such that
\beq
\label{last-est}
|V(b,\xi)|\le C'\lara{\xi}^{pq+\kappa}|V(a,\xi)|,
\eeq
for $|\xi|\ge 1$. This proves the $\Cinf$ well-posedness of the Cauchy problem \eqref{CP_intro}. Similarly, \eqref{last-est} implies the well-posedness of \eqref{CP_intro} in $\mathcal{D}'(\R^n)$.
\end{proof}

\end{document}